\documentclass[11pt,a4paper]{article}
\usepackage{amsmath}
\usepackage{amsthm}
\usepackage{amsfonts} 
\usepackage[pdftex]{graphicx} 
\usepackage[english]{babel} 
\usepackage[pdftex,linkcolor=black,pdfborder={0 0 0}]{hyperref} 
\usepackage{calc} 
\usepackage{enumerate}

\usepackage{graphics}
\usepackage{tikz}

\usetikzlibrary{positioning}
\allowdisplaybreaks
\usepackage[a4paper, lmargin=0.1666\paperwidth, rmargin=0.1666\paperwidth, tmargin=0.1111\paperheight, bmargin=0.1111\paperheight]{geometry} 
\usepackage[all]{nowidow}
\usepackage[protrusion=true,expansion=true]{microtype}
\usepackage{lipsum}

\theoremstyle{plain}
\newtheorem{theorem}{Theorem}[section]

\newtheorem{conjecture}[theorem]{Conjecture}

\newtheorem{lemma}[theorem]{Lemma}

\theoremstyle{definition}

\theoremstyle{remark}

\title{Non-left-orderability of cyclic branched covers of pretzel knots $P(3,-3,-2k-1)$}
\author{Lin Li \and Zipei Nie}

\AtEndDocument{\bigskip{\footnotesize
  \noindent
  \textsc{Lin Li. Huawei HiSilicon.}
    \par\noindent
  \textit{E-mail address}: \texttt{\href{mailto:lilin29@huawei.com}{lilin29@huawei.com}}
  
  \par
  \addvspace{\medskipamount}
  \noindent
  \textsc{Zipei Nie. Nine-Chapter Lab, Huawei.}  
  
  \par\noindent
  \textit{E-mail address}: \texttt{\href{mailto:niezipei@huawei.com}{niezipei@huawei.com}}
}}

\begin{document} 
\maketitle
\abstract{We prove the non-left-orderability of the fundamental group of the $n$-th fold cyclic branched cover of the pretzel knot $P(3,-3,-2k-1)$ for all integers $k$ and $n\ge 1$. These $3$-manifolds are $L$-spaces discovered by Issa and Turner.} 
	\section{Introduction}
	A nontrivial group $G$ is called left-orderable, if and only if it admits a total ordering $\le$ which is invariant under left multiplication, that is, $g\le h$ if and only if $fg\le fh$. An $L$-space is a rational homology $3$-sphere with minimal Heegaard Floer homology, that is, $\dim\widehat{HF} =|H_1(Y)|$. The following equivalence relationship is conjectured by Boyer, Gordon and Watson.

\begin{conjecture} \label{BGW-conj} \cite{BGW}
An irreducible rational homology $3$-sphere is an $L$-space if and only if its fundamental group is not left-orderable.
\end{conjecture}

Issa and Turner \cite{IT} studied a family of pretzel knots named $P(3,-3,-2k-1)$. They defined a family of two-fold quasi-alternating knots $L(k_1,k_2,\ldots, k_n)$, as shown in Figure \ref{fig:link}, and constructed the homeomorphism $\Sigma_n(P(3,-3,-2k-1)) \cong \Sigma_2(L(-k,-k,\ldots, -k))$. Because two-fold quasi-alternating links are Khovanov homology thin \cite{SS}, the double branched covers of them are $L$-spaces \cite{OS05}. In this way, they proved that all $n$-th fold cyclic branched covers $\Sigma_n(P(3,-3,-2k-1))$ are $L$-spaces.

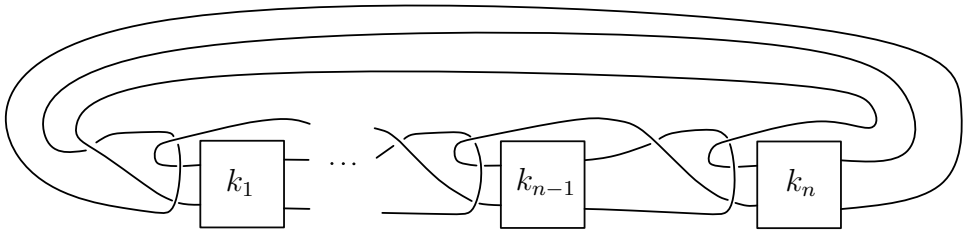
\begin{figure}[!htbp]\label{fig:link}
\begin{tikzpicture}
	\node[anchor=south west,inner sep=0] (image) at (0,0) {\includegraphics[width=0.9\textwidth]{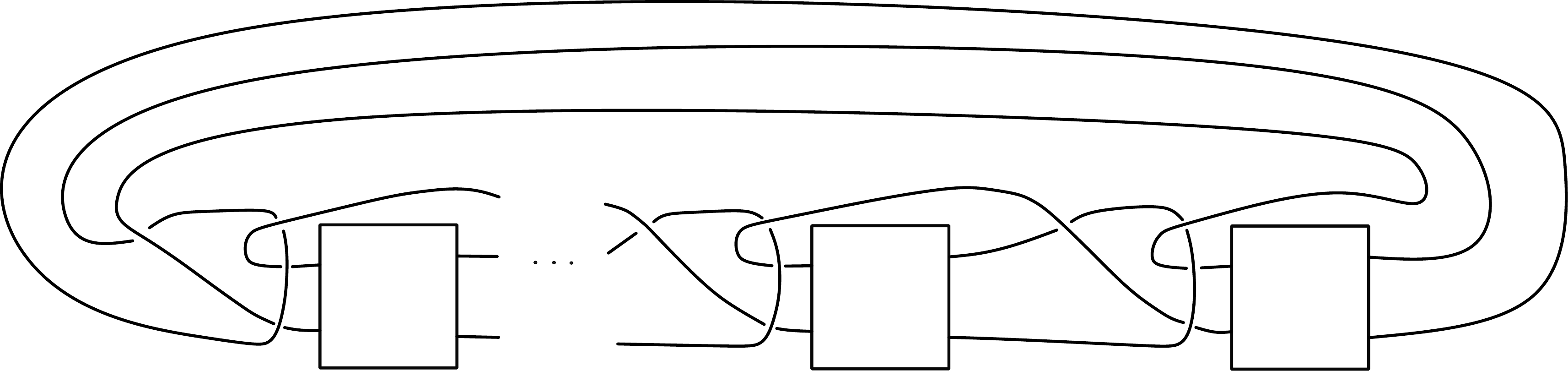}};
	\begin{scope}[x={(image.south east)},y={(image.north west)}]
        \node at (0.246,0.21) {\large $k_1$};
        \node at (0.565, 0.21){\large $k_{n-1}$};
        \node at (0.831, 0.21){\large $k_n$};
    \end{scope}
\end{tikzpicture}
\caption{The link $L(k_1,k_2,\ldots, k_n)$. This diagram is copied from \cite{IT}.}
\end{figure}

We prove the following result which is consistent with Conjecture \ref{BGW-conj}.

\begin{theorem}
    For any integers $k_1,k_2,\ldots, k_n$, the double branched cover of $L(k_1,k_2,\ldots, k_n)$ has non-left-orderable fundamental group.
\end{theorem}

As a corollary, the fundamental group of the $n$-th fold cyclic branched cover of the pretzel knot $P(3,-3,-2k-1)$ is not left-orderable.

In Section \ref{sec:2}, we derive a presentation of the fundamental group of double branched cover of $L(k_1,k_2,\ldots, k_n)$.

In Section \ref{sec:3}, we prove our main result.

\section*{Acknowledgement}
The authors thank Wei Chen, Nathan Dunfield, Heng Liao, Luigi Lunardon and Hannah Turner for helpful discussions.

\section{The Brunner's presentation}\label{sec:2}

A tool to compute the fundamental group of the double branched cover of an unsplittable link is the Brunner's presentation \cite{Brunner}. One could also use the Wirtinger's presentation to derive an equivalent form, but the computation would be longer. Historically, Ito \cite{Ito} used the coarse Brunner's presentation, a generalized version of Brunner's presentation, to prove the non-left-orderability of double branched covers of unsplittable alternating links. Abchir and Sabak \cite{AS} proved the non-left-orderability of double branched covers of certain kinds of quasi-alternating links in a similar way.

Consider the checkerboard coloring of the knot diagram $D$ in Figure \ref{fig:link}. Let $G$ and $\tilde{G}$ be the decomposition graph
and the connectivity graph. They can be chosen as shown in Figure \ref{fig:G}.

\begin{figure}[!htbp]\label{fig:G}
\centering\begin{tikzpicture}[node distance={20mm}, main/.style = {draw, circle}] 
\node[main] (1) {$1$}; 
\node[main] (2) [right of=1] {$2$}; 
\node[main] (3) [above right of=2] {$3$}; 
\node[main] (4) [above of=3] {$4$};  
\node[main] (5) [above left of=4] {$5$}; 
\node[main] (6) [left of=5] {$6$}; 
\node[main] (7) [below left of=6] {$7$}; 
\node[main] (8) [below of=7] {$8$}; 
\node[main] (0) [above right=18mm and 5mm of 1]{$0$};
\draw[-] (1) -- node[midway, above , sloped] {$k_1$} (2);
\draw[-] (2) -- node[midway, above , sloped] {$+1$} (3);
\draw[-] (3) -- node[midway, above , sloped] {$k_2$} (4);
\draw[-] (4) -- node[midway, above , sloped] {$+1$} (5);
\draw[-] (5) -- node[midway, above , sloped] {$k_3$} (6);
\draw[-] (6) -- node[midway, above , sloped] {$+1$} (7);
\draw[-] (7) -- node[midway, above , sloped] {$k_4$} (8);
\draw[-] (8) -- node[midway, above , sloped] {$+1$} (1);
\draw[-] (0) -- node[midway, above , sloped] {$-2$} (1);
\draw[-] (0) -- node[midway, above , sloped] {$+1$} (2);
\draw[-] (0) -- node[midway, above , sloped] {$-2$} (3);
\draw[-] (0) -- node[midway, above , sloped] {$+1$} (4);
\draw[-] (0) -- node[midway, above , sloped] {$-2$} (5);
\draw[-] (0) -- node[midway, above , sloped] {$+1$} (6);
\draw[-] (0) -- node[midway, above , sloped] {$-2$} (7);
\draw[-] (0) -- node[midway, above , sloped] {$+1$} (8);
\end{tikzpicture} 
\begin{tikzpicture}[node distance={20mm}, main/.style = {draw, circle}] 
\node[main] (1) {$1$}; 
\node[main] (2) [right of=1] {$2$}; 
\node[main] (3) [above right of=2] {$3$}; 
\node[main] (4) [above of=3] {$4$};  
\node[main] (5) [above left of=4] {$5$}; 
\node[main] (6) [left of=5] {$6$}; 
\node[main] (7) [below left of=6] {$7$}; 
\node[main] (8) [below of=7] {$8$}; 
\node[main] (0) [above right=18mm and 5mm of 1]{$0$};
\draw[->] (2) -- node[midway, above , sloped] {$e_1$} (1);
\draw[->] (2) -- node[midway, above , sloped] {$b_4$} (3);
\draw[->] (4) -- node[midway, above , sloped] {$e_4$} (3);
\draw[->] (4) -- node[midway, above , sloped] {$b_3$} (5);
\draw[->] (6) -- node[midway, above , sloped] {$e_3$} (5);
\draw[->] (6) -- node[midway, above , sloped] {$b_2$} (7);
\draw[->] (8) -- node[midway, above , sloped] {$e_2$} (7);
\draw[->] (8) -- node[midway, above , sloped] {$b_1$} (1);
\draw[->] (0) -- node[midway, above , sloped] {$f_1$} (1);
\draw[->] (0) -- node[midway, above , sloped] {$g_1$} (2);
\draw[->] (0) -- node[midway, above , sloped] {$f_4$} (3);
\draw[->] (0) -- node[midway, above , sloped] {$g_4$} (4);
\draw[->] (0) -- node[midway, above , sloped] {$f_3$} (5);
\draw[->] (0) -- node[midway, above , sloped] {$g_3$} (6);
\draw[->] (0) -- node[midway, above , sloped] {$f_2$} (7);
\draw[->] (0) -- node[midway, above , sloped] {$g_2$} (8);
\end{tikzpicture} 
\caption{The decomposition graph $G$ and the connectivity graph $\tilde{G}$ for $n=4$.}
\end{figure}
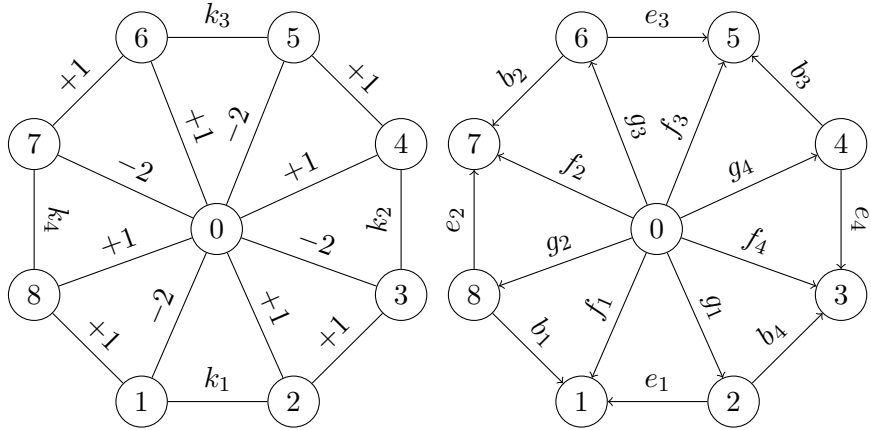

In the Brunner’s presentation of $\pi_1(\Sigma_2(L(k_1,k_2,\ldots, k_n)))$. Let $e_i, f_i, g_i,b_i$ ($1\le i\le n$) be the edge generators, and $a_i, c_i$ ($1\le i\le n$) be the region generators. Then the local edge relations are $e_i=a_i^{k_i}$, $b_i=c_i^{-1}$, $f_i=(a_{i}^{-1} c_i)^{-2}$ and $g_i=c_{i-1}^{-1} a_i$ ($1\le i\le n$). And the global cycle relations are $f_i^{-1} e_i g_i=1$ and $f_{i}^{-1} b_i g_{i+1}=1$ ($1\le i\le n$). Here the subscripts are considered modulo $n$.

By simplification, we get the following presentation of $\pi_1(\Sigma_2(L(k_1,k_2,\ldots,k_n)))$:	
$$\left \langle a_i,b_i| a_{i+1}=b_i^{-1} a_i  b_i a_i, a_i^{k_i}=b_i a_i b_i b_{i-1}^{-1}\right\rangle,$$
where $i=1,\ldots,n$ and we view subscripts modulo $n$.

	\section{The non-left-orderability}\label{sec:3}
We are going to prove our main non-left-orderability result. Let $\le$ be any left order on the group $\pi_1(\Sigma_2(L(k_1,k_2,\ldots,k_n)))$, then we have the following lemmas.

\begin{lemma}\label{l-1}
If $\forall m.\left(b_i^{-1} a_{i+1}^{m}\ge 1\right),$ then $\forall m_2\exists m_1 .\left(a_{i+1}^{-m_1} b_i^{-1} a_{i}^{m_2}\ge 1\right).$
\end{lemma}
\begin{proof}
By assumption, we have 
$$b_i^{-1}\ge 1.$$
Suppose, for the sake of contradiction, that for some integer $m_2$, 
$$\forall m_1.\left(a_{i+1}^{-m_1} b_i^{-1} a_{i}^{m_2}\le 1\right),$$ 
then we have $$ b_i^{-1} a_{i}^{m_2}\le 1.$$
Since $a_{i+1}=b_i^{-1} a_i  b_i a_i$, we have $$\left(a_{i+1} a_i^{-1}\right)^{m_2} b_i^{-1}= b_i^{-1} a_i^{m_2}.$$
Hence we have $$a_i^{m_2} \le 1,\left(a_{i+1}a_i^{-1}\right)^{m_2} \le 1,$$
which implies $$a_{i+1}^{m_2}\le 1.$$
By assumption, we have
$$\forall m.\left(a_{i+1}^{-m} a_i^{m_2}\le 1\right).$$
So we have
$$\forall \left(m \mbox{ has same sign as } m_2\right).\left(a_{i+1}^{-m} a_i^{m_2}a_{i+1}^{m} \le 1\right).$$
If $m_2\ge 0$, then 
$$\left(a_{i+1} a_i^{-1}\right)^{m_2} = a_{i+1}^{m_2} \left(a_{i+1}^{-m_2+1}a_i^{-1} a_{i+1}^{m_2-1}\right)\cdots\left(a_{i+1}^{-1}a_i^{-1} a_{i+1}\right) a_i^{-1}.$$
If $m_2<0$, then 
$$\left(a_{i+1} a_i^{-1}\right)^{m_2}= a_{i+1}^{m_2} \left(a_{i+1}^{-m_2}a_i a_{i+1}^{m_2}\right)\cdots\left(a_{i+1}^{2}a_i a_{i+1}^{-2}\right)\left(a_{i+1}a_i a_{i+1}^{-1}\right).$$
Therefore, we have 
$$a_{i+1}^{-m_2}\left(a_{i+1} a_i^{-1}\right)^{m_2} \ge 1.$$
So we have 
$$ a_{i+1}^{-m_2} b_i^{-1} a_{i}^{m_2}=a_{i+1}^{-m_2}\left(a_{i+1} a_i^{-1}\right)^{m_2} b_i^{-1} \ge 1,$$
which implies the conclusion.
\end{proof}
\begin{lemma}\label{l-2}
If $\forall m.\left(b_i^{-1} a_{i+1}^{m}\ge 1\right),$ then $\forall m.\left(b_{i-1}^{-1} a_{i}^{m}\ge 1\right).$
\end{lemma}
\begin{proof}
Because $a_{i+1}=b_i^{-1} a_i  b_i a_i$ and $a_i^{k_i}=b_i a_i b_i b_{i-1}^{-1}$, we have
$$b_{i-1}^{-1} a_i^{m}= b_i^{-1} a_i^{-1} b_i^{-1} a_i^{m+k_i}= b_i^{-1} a_{i+1}^{-1} b_i^{-1} a_i^{m+k_i+1}.$$
By Lemma \ref{l-1}, we have $$\exists m_1.\left(a_{i+1}^{-m_1} b_i^{-1} a_i^{m+k_i+1} \ge 1\right).$$
By assumption, we have $$b_i^{-1} a_{i+1}^{m_1-1}\ge 1.$$
Therefore $$b_{i-1}^{-1} a_i^{m}=\left(b_i^{-1} a_{i+1}^{m_1-1}\right)\left(a_{i+1}^{-m_1} b_i^{-1} a_i^{m+k_i+1}  \right)\ge 1.$$
\end{proof}

Let us apply the fixed point method on $a_1$. This is a common technique in the proofs of many non-left-orderability results. If the group $\pi_1(\Sigma_2(L(k_1,k_2,\ldots, k_n)))$ has a left order $\le$, then \cite{Holland} there exists a homomorphism $\rho$ from this group to $\mbox{Homeo}_+(\mathbf{R})$ with no global fixed points, and $g\le h$ if and only if $\rho(g)(0)\le \rho(h)(0)$. Then for the element $a_1$, there are two situations: 

\begin{enumerate}[(1)]
    \item If $\rho(a_1)$ has a fixed point $s$, then there is a left total preorder $\le_{a_1}$, defined as $g\le h$ if and only if $\rho(g)(s)\le \rho(h)(s)$, with $a_1\ge_{a_1} 1$ and $a_1^{-1} \ge_{a_1} 1$.
    \item Otherwise, any conjugate of $a_1$ has the same sign as $a_1$.
\end{enumerate}

Let us assume the first situation. Without loss of generality, we assume that $b_n \le_{a_1} 1$, then $\forall m.\left(b_n^{-1} a_{1}^{m}\ge_{a_1} 1\right).$ Since the antisymmetry is never used in the proofs of Lemma \ref{l-1} and Lemma \ref{l-2}, they apply to $\le_{a_1}$. By inductive application of Lemma \ref{l-2}, we have $\forall m.\left(b_{i}^{-1} a_{i+1}^{m}\ge_{a_1} 1\right)$ for any $i=1,\ldots,n$. By Lemma \ref{l-1}, we have the relation $\forall m_2\exists m_1 .\left(a_{i+1}^{-m_1} b_i^{-1} a_{i}^{m_2}\ge_{a_1} 1\right)$ for any $i=1,\ldots,n$.

Because $a_{i+1}=b_i^{-1} a_i  b_i a_i$ and $a_i^{k_i}=b_i a_i b_i b_{i-1}^{-1}$, we have
$$b_i^{-1} a_i^{k_i}= a_i b_i b_{i-1}^{-1}=b_i a_{i+1}a_i^{-1} b_{i-1}^{-1}= (b_i a_{i+1})(b_{i-1}a_i)^{-1}.$$
Hence we have $$\forall m_2\exists m_1 .\left( (b_{i-1}a_i)^{-1} a_{i}^{m_2}\ge_{a_1} (b_i a_{i+1})^{-1} a_{i+1}^{m_1}\right).$$
By induction, we have $$\forall m_2\exists m_1 .\left( (b_{1}a_2)^{-1} a_{2}^{m_2}\ge_{a_1} (b_n a_{1})^{-1} a_{1}^{m_1}\right),$$
especially
$$\exists m_1 .\left( (b_{1}a_2)^{-1} \ge_{a_1} (b_n a_{1})^{-1} a_{1}^{m_1}\right).$$
Then we have 
$$\exists m_1 .\left( 1 \ge_{a_1} b_1^{-1} a_1^{m_1+k_1}\right),$$
which implies $b_1\ge_{a_1} 1$. Remember that $\forall m.\left(b_{i}^{-1} a_{i+1}^{m}\ge_{a_1} 1\right)$ implies $b_{1}^{-1}\ge_{a_1} 1$. 

Therefore every fixed point of $a_1$ is also a fixed point of $b_1$. By $a_{i+1}=b_i^{-1} a_i  b_i a_i$, every fixed point of $a_1$ is also a fixed point of $a_2$. By symmetry and induction, any fixed point of $a_1$ is a global fixed point.

Now we assume any conjugate of $a_1$ is positive. By $a_{i+1}=b_i^{-1} a_i  b_i a_i$, if any conjugate of $a_i$ is positive, then any conjugate of $a_{i+1}$ is positive. By induction, any conjugate of $a_i$ is positive. By $a_{i+1}a_i^{-1}=b_i^{-1} a_i  b_i \ge 1$ and their product is identity, we have $a_i=a_{i+1}$ for all $i=1,\ldots, n$. Furthermore, we have $a_i=1$. Then the fundamental group is finite, so it is not left-orderable.

Therefore, we proved the non-left-orderability.

\end{document}